\documentclass[11pt]{article}
\usepackage{amsmath,amsfonts,amsthm,amscd,amssymb,graphicx}

\numberwithin{equation}{section}
\newtheorem{theorem}{Theorem}[section]

\newtheorem{lemma}[theorem]{Lemma}

\newtheorem{proposition}[theorem]{Proposition}

\newtheorem{remark}[theorem]{Remark}

\def\bR  {\mathbb{R}}
\def\del{\partial }
\def \la {\langle_x}
\def \ra {\rangle}

\def\bR  {\mathbb{R}}

\def\del{\partial }
\def \la {\langle}
\def \ra {\rangle_x}

\def\thh{\theta_h}

\date{July 10, 2017}

\begin{document}

\title{Onsager's Conjecture for the Incompressible Euler Equations in  Bounded Domains}
\author{Claude Bardos\footnotemark[1]  \and Edriss S.\, Titi%
\footnotemark[2]   }
\maketitle

\begin{abstract}

The goal of this note is to show that, also in a bounded domain $\Omega \subset \bR^n$\,, with $\partial \Omega\in C^2$\,, any weak solution, $(u(x,t),p(x,t))$\,, of the Euler equations of ideal incompressible fluid in   $\Omega\times (0,T) \subset \bR^n\times\bR_t$\,, with the impermeability boundary condition: $u\cdot \vec n =0$ on $\del\Omega\times(0,T)$,  is of constant energy on the interval $(0,T)$ provided the velocity field $u \in L^3((0,T); C^{0,\alpha}(\overline{\Omega}))$, with $\alpha>\frac13\,.$

%This paper aims first to justify the Maxwell-Boltzmann approximation for
%electrons with an electron temperature given by an energy conservation
%argument; it is made by an analysis of the stationary solution of the electron
%Vlasov-Boltzmann equation. Secondly we state a reduced ion kinetic model using this Maxwell-Boltzmann
%approximation and we prove its well-posedness.

\end{abstract}

\tableofcontents

\renewcommand{\thefootnote}{\fnsymbol{footnote}}

\footnotetext[1]{%
Laboratoire J.-L. Lions, BP187, 75252 Paris Cedex 05, France. Email:
claude.bardos@gmail.com}

\footnotetext[2]{%
Department of Mathematics,
                 Texas A\&M University, 3368 TAMU,
                 College Station, TX 77843-3368, USA. Also, Department of Computer Science and Applied Mathematics, The Weizmann Institute of Science, Rehovot 76100, ISRAEL. Email: titi@math.tamu.edu \, and \, edriss.titi@weizmann.ac.il}

\section{Introduction and preliminary  remarks}
The aim of  this article  is to prove the following:
\begin{theorem} \label{basic}
Let $\Omega\subset \bR^n$ be a bounded domain with  $C^2$  boundary, $\del\Omega$; and let  $(u(x,t),p(x,t))$ be a weak solution of the incompressible Euler equations in $\Omega\times(0,T)$, i.e.,
%\begin{subequations}
\begin{equation}
 u\in L^\infty((0,T),L^2(\Omega))\,, \quad \nabla \cdot u=0\quad  \hbox{in} \quad \Omega\times(0,T)\,,\,   \hbox{and}  \quad u\cdot \vec n =0 \quad \hbox{on} \quad \del \Omega\times(0,T)\,,
 \end{equation}
%\begin{equation}
%\nabla \cdot u=0 \quad \hbox{ in} \quad  \Omega\quad
% \end{equation}
and for every test  vector field  $ \Psi (x,t)\in \mathcal D(\Omega \times (0,T)):$
\begin{equation}
\la u ,\del_t\Psi\ra + \la u\otimes u : \nabla \Psi\ra + \la p,\nabla \cdot \Psi\ra=0\label{Eulerdistribution}\,, \quad \hbox{in} \quad L^1(0,T)\,.\\\
  \end{equation}
%\end{subequations}
 %\end{document}
Assume that
\begin{equation}
u\in L^3((0,T); C^{0,\alpha}(\overline{\Omega})),
\end{equation} with $\alpha > \frac13$, then the  energy conservation holds true, that is:
\begin{equation}
\|u(.,t_2)\|_{L^2(\Omega)}= \|u(., t_1)\|_{L^2(\Omega)}\,, \quad \hbox{for every} \quad t_1, t_2 \in (0,T)\,.
\end{equation}
\end{theorem}
In the above statement   $\la \cdot, \cdot \ra$ denotes the distributional duality with respect to the spatial variable $x$. For the justification of the weak formulation, as it is stated in the above theorem, see, e.g., Lions and Magenes \cite{Lions-Magenes} page 8, and Schwartz \cite{Schwartz}.

Notably, this theorem implies that to dissipate energy a weak solution of Euler equations must not be in the space more regular than $L^3((0,T); C^{0,\alpha}(\overline{\Omega}))$\,, with $\alpha>\frac 13\,.$   Such fact was observed, with a formal proof by Onsager in 1949 \cite{ON}. Hence it carries the name of Onsager conjecture.  In the absence of physical boundary (i.e., in the whole space $\Omega =\bR^d$ or for the case of periodic boundary conditions  in the torus $\Omega= \mathbb{T}^n$) this conjecture was proven in 1994 by Constantin, E and Titi \cite{CET}, after a first preliminary result of  Eyink \cite{GEY} (see also \cite{Cheskidov-etal}).
Moreover, the relevance of this issue has been underlined  by a series of contributions
(cf.\, Isett \cite{Isett}, Buckmaster, De Lellis ,  Sz\'ekelyhidi and  Vicol  \cite{BCDSV} and references therein) where  weak solutions,   $u\in C^{0,\alpha} ((0,T) ; (\mathbb{T}^n)$, with $\alpha <\frac 13$, that  dissipate energy were constructed.  These  results are concerning the problem in domains without physical  boundaries. However, due to the well recognized  dominant role of the boundary in the generation of turbulence (cf.\, \cite{BT} and references therein) it seems very reasonable  to investigate the analogue of the Onsager conjecture in bounded domains. Eventually, the  need to localize in order to deal with the boundary effect, as it will be shown below,  stimulates the construction of a direct proof which may have further applications.

%%%%%%%%%%%%
%%%%%%%%%%%
%%%%%%%%%%%
The  proof of the theorem will consist of several fundamental steps presented in the following propositions.

%%%%%%%%%%%%%%%%%%%%%%%%
%%%%%%%%%%%%%%%%%%%%%
%%%%%%%%%%%%%%%%%%%%%
\begin{proposition} \label{pressure10}
Under the assumptions of Theorem \ref{basic} the pair $(u,p)$ satisfies the following regularity properties
\begin{subequations}
\begin{equation}
u\otimes u\in L^3((0,T); L^2(\Omega) )\,, \quad p\in L^{\frac 32} ((0,T); C^{0,\alpha} (\overline{\Omega}) )\,, \label{holder}
\end{equation}
\begin{equation}
\del_t u =-\nabla \cdot(u\otimes u) -\nabla p \in L^{\frac32} ((0,T) ;H^{-1}(\Omega) )\,.\label{holder2}
\end{equation}
\end{subequations}
\end{proposition}
%%%%%%%%%%%%%%%%%%%%%%%%%
%%%%%%%%%%%%%%%%%%%%%
%%%%%%%%%%%%%%%%%%%%%
%\end{document}
\begin{proof}
The first part of  (\ref{holder}) is an immediate consequence of the assumption that $u\in L^\infty((0,T);L^2(\Omega))\cap L^3((0,T); C^{0,\alpha}(\overline{\Omega}))$. For the second part  of (\ref{holder}) we first observe that, from the definition of weak solutions of the Euler equations,  the pressure, $p$, is a solution of the following elliptic boundary-value problem:
\begin{equation}
  -\Delta p=\sum_{i,j=1}^n \del_{x_i}\del_{x_j} (u_i u_j)\, \quad \hbox{in} \quad \Omega, \quad \hbox{and}  \quad \frac{\del p}{\del \vec{n}}=-\sum_{i,j=1}^n u_ju_j \del_{x_i}\vec n_j \quad \hbox{on}  \quad\del \Omega \label{holder3}\,.
\end{equation}
Observe that the boundary condition in  (\ref{holder3})  follows from simple  calculations for the case of classical solutions using the fact that $u\cdot \vec n =0$ (see, e.g., \cite{Temam}), which is considered here to be the suitable boundary condition in the definition of weak solution for the pressure.
Applying the classical theory of elliptic equations in H\"older spaces applied to (\ref{holder3}) (cf., e.g., \cite{Krylov} chapters 5 and 6) implies the estimate
\begin{equation}
\|p(.,t)|\|_{C^{0,\alpha}}\le C\|u(.,t)\|_{C^{0,\alpha}}^2\,, \label{pressure0}
\end{equation}
from which one infers the second part of (\ref{holder}).
Eventually, (\ref{holder2}) follows from (\ref{Eulerdistribution}) and  (\ref{holder})\,.  \end{proof}

%%%%%%%%%%%
%%%%%%%%%%%%%%
%\subsection{Localization and regularization}
To investigate the boundary effect one introduces the distance to the boundary:
\begin{equation*}
\hbox{for any}\, x \in \overline{\Omega}, \,\, d(x)=\inf_{y\in\del \Omega} |x-y|,\,\hbox{ and the open set} \quad  \Omega_h= \{ x\in \Omega|\, \hbox{such that} \, d(x)<h\}\,.
\end{equation*}
 Since $\del\Omega$ is assumed to be a  $C^2$ compact manifold,  there exists  $h_0(\Omega)>0$ with the following properties (for an explicit construction see, e.g., \cite{GuilleminSternberg} page 9):

 1 For any $x \in \overline{\Omega_{h_0}} $, the function $x\mapsto d(x)$ belongs to $C^1(\overline{\Omega_{h_0}})\,;$

 2 for any $x\in \overline{\Omega_{h_0}}$ there exists a unique point $\sigma(x)\in\del \Omega$ such that
 \begin{equation}\label{distance}
 d(x)= |x-\sigma(x)| \,\quad\hbox{and one has} \quad  \nabla d(x) =- \vec n(\sigma(x))\,.
 \end{equation}
Then one introduces  a $C^\infty(\bR)$    nondecreasing function $\eta: \bR\mapsto [0,1]$, with $\eta(s) =0$, for $s\in (-\infty,\frac{1}{2}]$, and $\eta(s) =1$, for $s\in [1,\infty)$\,.  For $h\in (0,h_0)$ the function $\theta_h(x)=\eta(\frac{d(x)}h)\,,$  is compactly supported $C^1(\Omega)$ function. We will also denote by $\thh$ its extension,  by zero, outside $\Omega$. Similarly, for any $w \in L^\infty (\Omega)$ the compactly supported function  $\thh w $  is well defined in $\Omega$, and its extension,  by zero outside $\Omega$, is also well defined over all $\bR^n$, and  will be also denoted by  $\thh w$.
%defines  with $s=d (x,\del\Omega)$ a $C^2$ isomorphism of $\del\Omega\times(0,h)$ on $\Omega_h=\{x\in \Omega : d(x,\del \Omega)<h\}\,.$
%Then with $\eta:\br\mapsto[0,1]$ aone defines for $h\le h_0 $ the function
%\begin{equation}
% \theta:\Overline\Omega \rightarrow \bR:  \theta_h(x)=\eta(\frac{d(x)}h)
%\end{equation}
Next, one has the following :
\begin{lemma} \label{bdycon} Let $h \in (0,h_0)$. For any vector field  $w\in C^{0,\alpha}(\overline \Omega)$, with $w \cdot \vec n =0$ on $\del\Omega\,,$  one has  the following estimates (with a constant $C$ independant of $h$, but might depend on $\Omega$):
\begin{subequations}
\begin{equation}
| w(x) \cdot \nabla \thh(x) |\le C\|w\|_{C^{0,\alpha}(\Omega)} h^{\alpha-1}  \label{orth0}, \quad \hbox{for every}\quad x\in \bR^n,
\end{equation}
\begin{equation}
  \quad \int_{\bR^n} |w(x) \cdot \nabla \thh(x)|dx \le C\|w\|_{C^{0,\alpha}(\Omega)} h^{\alpha} \label{orth}\,.
\end{equation}
\end{subequations}
\end{lemma}

% \end{document}
%one introduces the regularising function
% $\ and The parameter $h$ corresponds to a localisation away from the boundary and The parameter $\epsilon$ to a regularization by
%
%One $\del\Omega$ is a bounded $C^2$ manifold,
\begin {proof} Observe that $w(x) \cdot \nabla \theta_h(x)=0$,  for every $x\in (\Omega_h)^c$. Moreover, for $x \in\Omega_h $, thanks to (\ref{distance}),  one has:
\begin{equation}
\nabla \theta_h(x)=-\frac 1h\eta'(\frac{d(x)}h) \vec n(\sigma(x))\,.
\end{equation}
Then for every $x\in \Omega_h $, we use the fact that  $w(\sigma(x))\cdot \vec n(\sigma(x)=0,$ to obtain:
\begin{equation}
 \begin{aligned}
&|w(x) \cdot \nabla \thh(x)| =\frac 1h\eta'(\frac{d(x)}h)|(w(x)-w( \sigma(x))\cdot \vec n(\sigma(x) |\\
&\quad \quad\quad\quad \quad\le\frac C h \|w\|_{C^{0,\alpha}}|x-\sigma(x)|^\alpha \le C\|w\|_{C^{0,\alpha}} h^{\alpha-1}\,.
 \end{aligned}
\end{equation}
Combining all the above we conclude (\ref{orth0}). Estimate (\ref{orth}) follows by integrating (\ref{orth0}) over $\bR^n$ taking into account the facts  that the support of  $\nabla\thh $ is a subset of $\overline{\Omega_h}$\,, and that $|\Omega_h| \le C h$.
\end {proof}
As in  \cite{CET}, we introduce a nonnegative radially symmetric $C^\infty (\bR^n)$ mollifier, $\phi(x) $, with support in $|x|\le 1$, and $\int_{\bR^n} \phi(x) dx=1$. Furthermore, for any $ \epsilon >0$, we denote by $\phi_\epsilon=\frac 1{\epsilon^n}\phi(\frac x\epsilon)$, and by  $v^\epsilon= v\star \phi_\epsilon$, for any $v\in\mathcal D'(\bR^n)$. Moreover, for $h\in (0,h_0)$, the distributions  $(\thh v)^\epsilon$  and  $((\thh v)^\epsilon)^\epsilon$  belong to $\mathcal D(\bR^n)\,;$ in addition,  they are compactly supported inside $\Omega,$	whenever  $\epsilon \in (0, \frac h 4)$.

\section{Fundamental steps toward proving energy conservation}
In this section we work under the assumptions of Theorem \ref{basic}, and we assume, all along,  that the regularization parameters $h$ and $\epsilon$ satisfy $h\in (0,h_0)$ and  $\epsilon \in (0,\frac h 4)$.  First observe that   by virtue of Proposition \ref{pressure10} equation (\ref {Eulerdistribution})  remains valid for test  vector field $\Psi \in W^{1,3}((0,T); H^1_0(\Omega))$. Therefore, we take  in (\ref {Eulerdistribution}) $\Psi = \thh((\thh u)^\epsilon)^\epsilon \in W^{1,3}((0,T); H^1_0(\Omega))$  to obtain:
\begin{equation}
   \la u ,\del_t(\thh((\thh u)^\epsilon)^\epsilon)\ra + \la u\otimes u :\nabla  (\thh((\thh u)^\epsilon)^\epsilon )\ra + \la p, \nabla \cdot  (\thh((\thh u)^\epsilon)^\epsilon) \ra  =0\label{Eulerdistribution2}\,,
 \end{equation}
 in $L^1(0,T)$\,.
 The last equation involves three terms:
 %\begin{equation}
% $\Psi$ by the  compactly supported regularized function $(\thh (\thh u)^\epsilon)^\epsilon)$ and analyzes the limit for $\epsilon\rightarrow 0$ and $h\rightarrow 0$ (under the hypothesis $\epsilon < ??$ of the three terms:
\begin{equation}
\label{three-terms}
J_1=  \la u ,\del_t(\thh((\thh u)^\epsilon)^\epsilon)\ra
\,,\, J_2 =\la u\otimes u : \nabla( \thh ((\thh u)^\epsilon)^\epsilon) \ra\,\,, \hbox{ and}\,  J_3=  \la p ,  \nabla \cdot (\thh((\thh u)^\epsilon)^\epsilon )\ra \,.
\end{equation}

For the term $J_1$ we have the following:
%%%%%%%%%%%%%%%%%%%%%
%%%%%%%%%%%%%%%%%%%%%%
%%%%%%%%%%%%%%%%%%%%%%%
\begin{proposition} \label{energy} Let  $u$ be as in Theorem \ref{basic}. Then for any $(t_1,t_2)\in (0,T)\,$ one has:
\begin{equation}
   \lim_{h \rightarrow  0}\int_{t_1}^{t_2}\la u, \del_t (\thh((\thh u)^\epsilon)^\epsilon)
\ra dt= \frac 12  \|u(t_2)\|^2_{L^2(\Omega)}- \frac 12  \|u(t_1)\|^2_{L^2(\Omega)}\label{energyf}
\end{equation}
\end{proposition}
%%%%%%%%%%%%%%%%%%%%%
%%%%%%%%%%%%%%%
%%%%%%%%%%%%%%%%%%
\begin{proof}
With the regularity estimates (\ref{holder}) and (\ref{holder2}) the duality between $L^3((0,T); H_0^1(\Omega))$ and $L^{\frac32}(0,T; H^{-1} (\Omega))$ gives:
 \begin{equation}
 \la u, \del_t(\thh((\thh u)^\epsilon)^\epsilon)
\ra= \la  (\thh  u)^\epsilon, \del_t (\thh u)^\epsilon \ra =\frac12\frac{d}{dt}\int_{\bR^n}|((\thh u)^\epsilon|^2dx\,, \quad \hbox{in}\quad L^1(0,T)\,,
\end{equation}
and the result follows, after integration in time,  from the Lebesgue Dominant Convergence Theorem and the fact that $\epsilon\in (0,\frac h4)$.
\end{proof}
%%%%%%%%%%%%%%%%
%%%%%%%%%%%%%%%%%%
For the  second term $J_2=   \la u\otimes u :\nabla(\thh ((\thh u)^\epsilon)^\epsilon) \ra  $ one has the  following:
%%%%%%%%%%%%%%%
%%%%%%%%%%%%
%%%%%%%%%%%%%%%%
\begin{proposition}\label{essential}
Let  $u$ be as in Theorem \ref{basic}. Then
  \begin{equation}
|J_2| = |  \la u\otimes u : \nabla(\thh ((\thh u)^\epsilon)^\epsilon) \ra | \le Ch^{\alpha}\|u\|_{C^{0,\alpha}}\|u\|_{L^\infty}^2 + C\|u\|_{C^{0,\alpha}}\epsilon^{\alpha-1}(\| u\|_{C^{0,\alpha}}\epsilon^\alpha +\| u\|_{L^\infty}\frac\epsilon h)^2 \label{essentialproof}
\end{equation}

\end{proposition}
\begin{proof}
One writes $J_2=J_{21}+J_{22}$ with
\begin{equation}
\begin{aligned}
&J_{21}= \la u\otimes u : (\nabla \thh) \otimes((\thh u)^\epsilon)^\epsilon\ra \\
&J_{22}= \la u\otimes u: \thh \nabla (((\thh u)^\epsilon)^\epsilon)\ra
\end{aligned}
\end{equation}
To estimate  the term $J_{21}$ one uses  Lemma \ref{bdycon} to obtain
\begin{equation}\label{impcon}
\begin{aligned}
|J_{21}|= |\la u\otimes u :  (\nabla \thh )\otimes ((\thh u)^\epsilon))^\epsilon \ra|
& = |\int_{\Omega_h}   ( u \cdot \nabla \thh(x)  )( u(x) \cdot ((\thh u)^\epsilon)^\epsilon) dx \\\
 %& \le \int_{\frac h 2\le d(x,\del\Omega) \le  h}  | ( (u(x)- u(\sigma(x)) \cdot \nabla \thh(x)  )( u(x) \cdot ((\thh u)^\epsilon)^\epsilon| dx\\
%&\int_{h\le d(x,\del\Omega)} \le 2h ( (u-u(\sigma(x))\cdot (\nabla \thh(x)  )( u(x) \cdot )^\epsilon))^\epsilon )dx\\
 %&\le  Ch^{\alpha-1}\|u\|_{C^{0,\alpha}}\|u\|_{L^\infty}^2 \int_{h\le d(x,\del\Omega)\le 2h}  dx\le
&\le Ch^{\alpha}\|u\|_{C^{0,\alpha}}\|u\|_{L^\infty}^2\,.
\end{aligned}
\end{equation}
Next, we  turn into estimating the term $J_{22}$.   First we observe that since  $u^\epsilon(x)$ is a divergence free smooth vector field for every $x\in \hbox{supp}\,(\thh u)^\epsilon \subset\subset \Omega$, therefore,   one has:
\begin{equation}
\la  (u^\epsilon \otimes (\thh u)^\epsilon): \nabla (\thh u)^\epsilon\ra= \int_\Omega (u^\epsilon \cdot \nabla (\thh u)^\epsilon ) \cdot (\thh u)^\epsilon\, dx  = 0.
\end{equation}
Consequently, one has the following estimate for $J_{22}$:
\begin{equation}
\begin {aligned}
|J_{22}| = & | \la u\otimes u : \thh \nabla ((\thh u)^\epsilon)^\epsilon\ra |
 = | \la (u\otimes \thh u) : \nabla ((\thh u)^\epsilon)^\epsilon \ra | = \\
& |\la (u\otimes \thh u)^\epsilon : \nabla (\thh u)^\epsilon \ra |  =
|\la\Big( (u\otimes \thh u)^\epsilon -(u^\epsilon \otimes (\thh u)^\epsilon)\Big):  \nabla (\thh u)^\epsilon \ra |.
\end{aligned}
\end{equation}
To treat the term
$$
\la\Big( (u\otimes \thh u)^\epsilon -(u^\epsilon \otimes (\thh u)^\epsilon)\Big) : \nabla (\thh u)^\epsilon\ra
$$
one uses similar computations to those in \cite{CET}  (cf. Remark \ref{reynold} below) which relate $ (u\otimes \thh u)^\epsilon$ to $(u^\epsilon \otimes (\thh u)^\epsilon)$. More precisely, for any two distributions, $v,w \in \mathcal{D}'(\bR^n)$, one has the following identity:
\begin{equation}
\begin{aligned}
&(v\otimes w)^\epsilon(x)-(v^\epsilon\otimes w^\epsilon)(x) = {\int_{\bR^n_y}( \delta_y v \otimes \delta _y w)(x) \phi_\epsilon(y)dy} + (v-v^\epsilon)(x)\otimes(w-w^\epsilon) (x)  \\
&\hbox{ with}\quad  (\delta_y v) (x) =v(x-y)-v(x)\,,\quad  \hbox{and} \quad (\delta_y w) (x) =w(x-y)-w(x) \,.
\end{aligned}
\end{equation}
%\end{proof}
Hence $J_{22}= J_{221} + J_{222} $ with:
\begin{equation}
\begin{aligned}
J_{221}&=\int_{\bR^n_x}\Big((\int_{\bR^n_y} (\delta_y u\otimes \delta _y (\thh u))(x)\phi_\epsilon(y)dy): ( \int_{\bR^n_z} \nabla \phi_\epsilon(z)\otimes(\thh u)(x-z)dz)\Big) dx \\
&=\int_{\Omega}\Big((\int_{\bR^n_y} (\delta_y u\otimes \delta _y (\thh u))(x)\phi_\epsilon(y)dy): ( \int_{\bR^n_z} \nabla \phi_\epsilon(z)\otimes(\thh u)(x-z)dz)\Big) dx
\end{aligned}
\end{equation}
and
\begin{equation}
\begin{aligned}
J_{222}&=\int_{\bR^n_x}\Big(((u-u^\epsilon) \otimes ((\thh u) -(\thh u)^\epsilon)): \nabla(\thh u)^\epsilon \Big)dx\\
&=\int_{\Omega}\Big( ((u-u^\epsilon) \otimes ((\thh u) -(\thh u)^\epsilon)): \nabla(\thh u)^\epsilon \Big)dx
\end{aligned}
\end{equation}
To estimate  $J_{221}$, first, one uses the facts that for every $|y| \le \epsilon$ one has $|(\delta_y \thh)(x)| \le C \frac{\epsilon}{h}$, and that the  $\hbox{supp} \, \phi_\epsilon \subset \{y| \,\, |y|\le \epsilon\}$,  together with the  $C^{0,\alpha}$ regularity of $u$ to obtain that:
\begin{equation}
\begin{aligned}
 &|\int_{\bR^n_y} (\delta_y u\otimes \delta _y (\thh u))(x)\phi_\epsilon(y)dy)|=
 |\int_{\bR^n_y} (\delta_y u) (x)\otimes (\thh (x-y)(\delta_y u)(x)+ (\delta_y\thh )(x) u(x-y))\phi_\epsilon(y)dy|\\
 &\le  C \epsilon^\alpha \|u\|_{C^{0,\alpha}} \int_{\bR^n_y}( \epsilon^\alpha \|u\|_{C^{0,\alpha}}+ \frac\epsilon h \|u\|_{L^\infty}) \phi_\epsilon(y)dy = C \epsilon^\alpha \|u\|_{C^{0,\alpha}} ( \epsilon^\alpha \|u\|_{C^{0,\alpha}}+ \frac\epsilon h \|u\|_{L^\infty})\,.
\end{aligned}
\end{equation}
Second,
\begin{equation}\label{grad_theta_u}
\begin{aligned}
&|\int_{\bR^n_z}\Big( \nabla \phi_\epsilon(z)\otimes(\thh u)(x-z)\Big)dz|=|\int_{\bR^n_z} \Big(\nabla \phi_\epsilon(z)\otimes((\thh u)(x-z)-(\thh u)(x))\Big)dz|\\
&=|\int_{\bR^n_z} \Big(\nabla \phi_\epsilon(z)\otimes(\delta_z\thh (x)u(x-z)+\thh (x)\delta_zu(x))\Big) dz|\\
%& \int_{\bR^n_z} \nabla_z \phi_\epsilon(z)\otimes(\delta_z\thh (x)u(x-z)+\thh (x)\delta_zu(x))dx \\
%&|\int_{\bR^n_z} \nabla_z \phi_\epsilon(z)\otimes(\delta_z\thh (x)u(x-z)+\thh (x)\delta_zu(x))dx|\\
&\le  C(\frac\epsilon h\|u\|_{L^\infty} + \epsilon^\alpha \|u\|_{C^{0,\alpha}})\int_{\bR^n_z}|\nabla\phi_\epsilon(z)|dz\le C\epsilon^{-1}(\frac\epsilon h\|u\|_{L^\infty} + \epsilon^\alpha \|u\|_{C^{0,\alpha}})\,,
\end{aligned}
\end{equation}
where in the last inequality we used the fact that $\int_{\bR^n_z}|\nabla\phi_\epsilon(z)|dz \le C \epsilon^{-1}$.
Hence from all the above one has:
\begin{equation}
 |J_{221}|\le C\epsilon^{\alpha-1}\|u\|_{C^{0,\alpha}} (\frac\epsilon h\|u\|_{L^\infty} + \epsilon^\alpha \|u\|_{C^{0,\alpha}})^2\,.
\end{equation}
To complete the proof of the Proposition \ref{essential},  it remains to estimate the term:
\begin{equation}
\begin{aligned}
J_{222}&=\int_{\bR^n_x}\Big( \Big((u-u^\epsilon) \otimes ((\thh u) -(\thh u)^\epsilon)\Big): \nabla(\thh u)^\epsilon \Big)dx\\
&=\int_{\Omega}\Big(\Big((u-u^\epsilon) \otimes ((\thh u) -(\thh u)^\epsilon)\Big): \nabla(\thh u)^\epsilon \Big) dx
\end{aligned}
\end{equation}
First, as in (\ref{grad_theta_u}) one has:
\begin{equation}
|\nabla(\thh u)^\epsilon(x)| \le C\epsilon^{-1}(\frac\epsilon h\|u\|_{L^\infty} + \epsilon^\alpha \|u\|_{C^{0,\alpha}})\,.
\end{equation}
Moreover, following similar arguments as in the above estimates for $J_{221}$ one can show that for every $x \in \hbox{supp}\, \thh$ one has
\begin{equation}
|(u-u^\epsilon)(x)|\le  \epsilon^\alpha \|u\|_{C^{0,\alpha}}  \quad\hbox{and} \quad
 |(\thh u)(x)-(\thh u)^\epsilon(x)|\le C( \epsilon^\alpha \|u\|_{C^{0,\alpha}}+ \frac\epsilon h\|u\|_{L^\infty})\,.
\end{equation}
Summing up, one has the following estimate for:
\begin{equation}
|J_{222} |\le C\epsilon^{\alpha-1} \|u\|_{C^{0,\alpha}} (\frac\epsilon h\|u\|_{L^\infty} + \epsilon^\alpha \|u\|_{C^{0,\alpha}})^2\,.
\end{equation}
Collecting the estimates on $J_2$ from $J_{21}$ and $J_{22}$ one obtains  (\ref{essentialproof}).
\end{proof}
%%%%%%%%%%%%%%
Eventually,  the introduction of the localized cutoff-function $\thh$ affects the divergence free property of the velocity field, $u$, of the solution $(u,p)$. Therefore, to estimate the term $J_3$ in (\ref{three-terms}), which involves the pressure, $p$,  one needs the following:
%%%%%%%%
\begin{proposition}\label{pressure1} Let $h\in (0,h_0)$ and $\epsilon \in (0,\frac{h}{4})$. Suppose $(u,p)$ is a  weak solution of the Euler equations  with    $u\in L^3((0,T);C^{0,\alpha}(\overline{\Omega}))$. Then one has the following estimate:
\begin{equation}
|\la p, \nabla\cdot (\thh((\thh u)^\epsilon)^\epsilon)\ra|  \le C \| u(t)\|^3_{C^{0,\alpha}}(h^\alpha+\epsilon^\alpha)\,.
\end{equation}
\end{proposition}
\begin{proof}
Thanks for Proposition \ref{pressure10} one can write:
\begin{equation}
\begin{aligned}
& \la p, \nabla\cdot \Big(\thh((\thh u)^\epsilon)^\epsilon\Big)\ra= \int_\Omega p \, \nabla\cdot \Big(\thh((\thh u)^\epsilon)^\epsilon\Big) dx=J_{31} +J_{32}\\
&\hbox{ with}\quad  J_{31}  =\int_{\Omega}(p \, \thh )\nabla\cdot ((\thh u)^\epsilon)^\epsilon dx  \quad \hbox{ and }  J_{32} =\int_{\Omega}p \, (\nabla \thh)\cdot ( (\thh u)^\epsilon)^\epsilon dx \,.
\end{aligned}
\end{equation}
For the term $J_{31}$ one obtains the following sequence of equalities by  integration by parts and successive use of the fact that $\nabla_x \phi_\epsilon(x-y)=-\nabla_y \phi_\epsilon(x-y)$:
\begin{equation}\label{J31}
\begin{aligned}
&J_{31}=\int_{\Omega}\Big ( (p(x)\thh(x))\nabla_x \cdot\Big(\int_{\bR^n_y}\int_{\bR^n_z} u(z)\thh(z) \phi_\epsilon(z-y)\phi_\epsilon(x-y)dz dy\Big)\Big)dx\\
&=\int_{\Omega} \Big(p(x)\thh(x) \int_{\bR^n_z}\int_{\bR^n_y} u(z)\thh(z) \phi_\epsilon(z-y)\cdot \nabla_x \phi_\epsilon(x-y) dy dz\Big)dx\\
&=-\int_{\Omega} \Big( p(x)\thh(x) \int_{\bR^n_z}\int_{\bR^n_y} u(z)\thh(z) \phi_\epsilon(z-y)\cdot \nabla_y \phi_\epsilon(x-y) dy dz\Big)dx\\
&=\int_{\Omega}\Big ( p(x)\thh(x) \int_{\bR^n_z}\int_{\bR^n_y} u(z)\thh(z) \phi_\epsilon(x-y) \cdot \nabla_y \phi_\epsilon(z-y) dy dz\Big)dx\\
&=-\int_{\Omega} p(x)\thh(x)( \int_{\bR^n_z}\int_{\bR^n_y} u(z)\thh(z) \phi_\epsilon(x-y) \cdot \nabla_z \phi_\epsilon(z-y) dy dz)dx\\
&=-\int_{\Omega} \Big (p(x)\thh(x) \int_{\bR^n_y} \int_{\bR^n_z} u(z) \phi_\epsilon(x-y) \cdot \Big(\nabla_z \Big( \thh(z)\phi_\epsilon(z-y)\Big) -\phi_\epsilon(z-y)\nabla\thh(z)\Big) dzdy \Big)dx\,.
\end{aligned}
\end{equation}
Observe that for every fixed $y \in \bR^n$,  the  function $\thh(z)\phi_\epsilon(z-y)$, as a function of $z$, is compactly supported in $\Omega$, and that there exists a sequence  $\chi_k(\cdot,y)\in \mathcal D(\Omega)$,  $k=1,2,\cdots$, such that
%%%%%%%%%%%%
\begin{equation}\label{approximation}
\lim_{k\to \infty} \|\chi_k(\cdot,y) - \thh(\cdot)\phi_\epsilon(\cdot-y)\|_{C^1(\Omega)}=0\,.
\end{equation}
%%%%%%%%%%
Therefore,  since $\nabla \cdot u =0$ in $\mathcal{D}'(\Omega)$, one has:
\begin{equation}\label{div}
\int_{\bR^n_z} u(z) \cdot \nabla_z \chi_k(z,y)  dz =0\,.
\end{equation}
Thus, for every fixed $y\in \bR^n$, by virtue of   (\ref{approximation})  and the fact that $u\in C^{0,\alpha}(\Omega)$
one  infers from (\ref{div}), by letting $k \to \infty$\,, that:
\begin{equation}\label{incompress}
\int_{\bR^n_z} u(z) \cdot \nabla_z ( \thh(z)\phi_\epsilon(z-y)) dz =0\,.
\end{equation}
Hence, as a result of (\ref{J31}) and  (\ref{incompress}) one has:
\[
J_{31} =\int_{\Omega}\Big( p(x)\thh(x) \int_{\bR^n_y} \int_{\bR^n_z} \phi_\epsilon(x-y) \phi_\epsilon(z-y) u(z) \cdot \nabla \thh(z) dzdy \Big)dx\,.
\]
Consequently, by virtue of Lemma \ref{bdycon} one has
\begin{equation}
|J_{31}|\le C \|p\|_{L^\infty}\|u\|_{C^{0,\alpha} }h^\alpha\,.
\end{equation}
Concerning the term $J_{32}$ observe again that the support of $\nabla \thh$ is contained in $\overline{\Omega_h}$, therefore, one has :
\begin{equation}
\begin{aligned}
&J_{32}= \int_{\Omega_h} \Big (p(x)\nabla\thh(x) \cdot \int_{\bR^n_z}\int_{\bR^n_y} \thh(x -y+z) u(x -y+z)\phi_\epsilon(y)\phi_\epsilon (z) dydz \Big ) dx\\
&=\int_{\Omega_h}  p(x)\Big ( \int_{\bR^n_y} \int_{\bR^n_z}\phi_\epsilon(y)\phi_\epsilon (z)\thh(x-y+z)\Big (u(x-y+z) -u(x)\Big ) \cdot \nabla\thh(x)  dydz \Big ) dx\\
&+ \int_{\Omega_h} p(x)\Big ( \int_{\bR^n_y} \int_{\bR^n_z} \phi_\epsilon(y)\phi_\epsilon (z) u(x)\cdot \nabla\thh(x)  dydz \Big ) dx =: J_{321} + J_{322}\,.
\end{aligned}
\end{equation}
In order to estimate the term $J_{321}$, one observes that for the relevant $x,y,z$ for which the integrand in the definition of    $J_{321}$ is not zero one has
$|(u(x-y+z) -u(x))|\le C \|u\|_{C^{0,\alpha}} \epsilon^{\alpha}$, and that $\int_{\Omega_h}|\nabla\thh (x)|dx <C$. As a result one obtains:
\begin{equation}
|J_{321}| \le C \|p
\|_{L^\infty} \|u\|_{C^{0,\alpha}} \epsilon ^\alpha.
\end{equation}
As for estimating $J_{322}$, Lemma \ref{bdycon} is used to obtain:
\begin{equation}
|J_{322} |\le \int_{\Omega_h} |p(x)|\int_{\bR^n_y} \int_{\bR^n_z} |u(x)\cdot \nabla\thh(x)| \phi_\epsilon(y)\phi_\epsilon (z)  dydz dx \le C\|p(x)\|_{L^\infty}\|u(t)\|_{C^{0,\alpha}} h^\alpha \,. \label{imbound2}
 \end{equation}

\end{proof}

Now, we are ready to complete the proof of Theorem \ref{basic}.
Let us integrate equation (\ref{Eulerdistribution2}) over the interval $(t_1,t_2) \subset (0,T)$ to obtain
\begin{equation}
\int_{t_1} ^{t_2} \la u ,\del_t(\thh((\thh u)^\epsilon)^\epsilon) \ra dt =- \int_{t_1}^{t_2}\la u\otimes u , \nabla (\thh((\thh u)^\epsilon)^\epsilon)\ra dt-  \int_{t_1}^{t_2}\la p,\nabla (\thh((\thh u)^\epsilon)^\epsilon)\ra dt
\end{equation}
At this stage we choose  $\epsilon = o(h^{\frac{2}{1+\alpha}})$, and since $\alpha > \frac {1}{3}$, then Theorem \ref{basic} follows  from Propositions \ref{pressure10}, \ref{energy}, \ref{essential} and \ref{pressure1} by letting $h\to 0$.

\section{Remarks}
\begin{remark}\label{reynold}
 The proof of Proposition \ref{essential} is an adaptation, to domain with boundary, of  the main argument of \cite{CET}. The proof involves the expression
$$
 ( (u\otimes \thh u)^\epsilon -(u^\epsilon \otimes (\thh u)^\epsilon)
$$
which is reminiscent of the Reynolds stress tensor as it appears in statistical theory of turbulence or in the vanishing viscosity weak limit of solutions of the Navier-Stokes equations, according to the formula:
\begin{equation}
\overline {(u_\epsilon \otimes v_\epsilon)}- {\overline u_\epsilon \otimes \overline v_\epsilon}=\overline{(\overline {u_\epsilon}-u_\epsilon)\otimes (\overline {v_\epsilon}-v_\epsilon))}\,.
\end{equation}
However, in the present work the  localization and regularization do not exactly behave as an average and this is the reason for the presence (both in \cite{CET}  and in this work) of the term
\begin{equation}
J_{221}= \int_{\bR^n_x}\Big(\Big (\int_{\bR^n_y} (\delta_y u\otimes \delta _y (\thh u))(x)\phi_\epsilon(y)dy\Big ): \Big (\int_{\bR^n_z} \nabla \phi_\epsilon(z)\otimes(\thh u)(x-z)dz \Big ) \Big ) dx
\end{equation}
which has to be estimated.
\end{remark}
\begin{remark}
 As expected the  impermeability boundary condition   ($u\cdot \vec n=0$ on  $\del\Omega$) plays an essential role in the arguments presented in this work. It is the main hypothesis in Lemma \ref{bdycon}, which is then  used for the estimation of $J_{21}$, in  formula (\ref{impcon}), and in the estimation of the pressure contribution term in formula (\ref{imbound2}).
 \end{remark}
 \begin{remark}
 Besides corresponding to physical situations that appear in nature, the introduction of boundary and boundary conditions is a stimulus for the construction of a direct proof avoiding, for instance, the use of Besov space. However, the arguments presented in this work may well be adapted for proving similar results while replacing the H\"older spaces $C^{0,\alpha}$ by some  ``exotic" function spaces. Moreover, the ideas introduced in this article may be also well adapted to consider the Onsager's conjecture for compressible fluids in bounded domains, hence extending some preliminary results  of  \cite{Gwiazda_et_al, Yu} and references therein.
 \end{remark}

% \end{document}
% Then observe that
% \begin{equation}
% J_2^{\epsilon,h} ={\rm{Trace}}\la u\nabla\ u, \thtea_n \nabla_x ((\theta_x((\theta_hu)^\epsilon)^\epsilon\ra={\rm{Trace}}\la (u\otimes \theta _h u)^\epsilon , \nabla_x ((\theta_x((\theta_hu)^\epsilon)^\epsilon\ra
% \end{equation}
% satisfies the following estimate:
% \begin{proposition}
% \|J_2^{\epsilon,h}\|
% \end{proposition}
% etc, etc
%\begin
% \appendix

% \section{}
% \label{}

% The Acknowledgements are an un-numbered section
 \section*{Acknowledgements}
  The authors are thankful to the kind and warm hospitality of the ICERM, Brown University,  where this work was initiated. C.B.\ is also thankful to the generous hospitality of the Weizmann Institute of Science where this work was completed.   The work of E.S.T.\ was supported in part by the ONR grant N00014-15-1-2333.


\begin{thebibliography}{99}
%\begin{thebibliography}
% please try to use the phi system -
% the references should be in alphabetical order of authors' names.
% Articles with a single author first, author will 1 co-author next,
% then author with several co-authors;

 \bibitem{BT} C.~Bardos and E.S.~Titi,   Mathematics and turbulence: where do we stand? {\it Jour. Turbulence}, {\bf 14(3)}, (2013), 42--76.

  \bibitem{BCDSV}  T. Buckmaster, C. De Lellis, L. Sz\'ekelyhidi Jr., and V. Vicol
    Onsager's conjecture for admissible weak solutions, arXiv:1701.08678  .


\bibitem{Cheskidov-etal} A.~Cheskidov, P.~Constantin, S.~Friedlander and R.~Shvydkoy, Energy con-
servation and Onsager’s conjecture for the Euler equations, {\it  Nonlinearity},
{\bf 21(6)}, (2008), 1233–-1252.

  \bibitem{CET} P.~Constantin, W.~E and E.S.~Titi, Onsager's Conjecture on the Energy Conservation for Solutions of Euler's equation, {\it Comm. Math. Phys.}, {\bf 165}, (1994), 207--209

\bibitem {GEY} G.L.~Eyink, Energy dissipation without viscosity in ideal hydrodynamics, I. Fourier analysis and local energy transfer, {\it Phys. D}, {\bf 78(3-4)}, (1994), 222--240.

\bibitem{Lions-Magenes} J.-L.~Lions and E.~Magenes,  Probl\'emes aux Limites Non Hom\'ogenes et Applications, Vol. 1, (1968), Dunod, Paris.

\bibitem {GuilleminSternberg} V.~Guillemin and S.~Sternberg, Geometric Asymptotic, American Mathematical Society, 1977.

\bibitem {Gwiazda_et_al} P.~Gwiazda, M.~Mich\'{a}lek and A.~ Swierczewska-Gwiazda, A note on weak solutions of coversation laws and energy/entropy conservation,  arXiv:1706.10154.

\bibitem{Isett} P.~Isett, A Proof of Onsager's conjecture, (2016),  arXiv:1608.08301.

\bibitem{Krylov} N.V.~Krylov, Lectures on Elliptic and Parabolic Equations in Holder Spaces. Graduate Studies in Mathematics Volume 12 . American Mathematical Society.

\bibitem{ON} L.~Onsager, Statistical hydrodynamics,  {\it Nuovo Cimento}, {\bf (9) 6}, Supplemento, 2(Convegno Internazionale di Meccanica Statistica)(1949), 279--287.


\bibitem{Schwartz} L.~Schwartz,
Th\'eorie des distributions \'a valeurs vectorielles. II.  {\it Ann. Inst. Fourier. Grenoble},  {\bf 8}, (1958), 1--209. (French)

\bibitem{Temam} R.~Temam, On the Euler equations of incompressible perfect fluids, {\it Journal of Functional Analysis}, {\bf 20}, (1975), 32--43.

\bibitem{Yu} C.~Yu, Energy conservation for the weak solutions of the compressible Navier-Stokes equations,   Arch. Rational Mech. Anal., (to appear),  (2017).

\end{thebibliography}
\end{document}